\documentclass[11pt]{article}
 \usepackage{amsfonts,amssymb,amsthm,amsmath,cite,bbm}

	\newcommand{\R}{\mathbb R} \newcommand{\N}{\mathbb{N}}
	\newcommand{\Rn}{\R^n}

\newcommand{\umin}{u_{\min}}

\newcommand{\Sph}{{\mathbb S}^{n-1}}

\newcommand{\e}{\varepsilon}

\newcommand{\conv}{\operatorname{conv}}
\newcommand{\elim}{\operatorname{epi-lim}}
\newcommand{\relint}{\operatorname{relint}}
\newcommand{\Int}{\operatorname{int}}
\newcommand{\dom}{\operatorname{dom}}
\newcommand{\epi}{\operatorname{epi}}
\newcommand{\pos}{\operatorname{pos}}

\newcommand{\Ind}{\mathrm{I}}

\newcommand{\infconv}{\mathbin{\Box}}

\newcommand{\eto}{\stackrel{epi}{\longrightarrow}}
		
\newcommand{\cK}{{\mathcal K}} \newcommand{\cP}{{\mathcal P}} \newcommand{\cL}{{\mathcal L}} 			\newcommand{\cS}{{\mathcal S}}

\newcommand{\oZ}{\operatorname{Z}}

\newcommand{\CV}{\operatorname{Conv}(\Rn)}
\newcommand{\CVpa}{\operatorname{Conv_{p.a.}}(\Rn)}
\newcommand{\C}[1]{\operatorname{Conv}(\R^{#1})}

 \renewcommand{\>}{\rangle}

\renewcommand{\l}{\ell}
\renewcommand{\d}{\,\mathrm{d}}

\newcommand\sln{\operatorname{SL}(n)}

	\newcommand{\eqnref}[1]{(\ref{#1})} 
	\newtheorem*{Theorem}{Theorem}
	\newtheorem{theorem}{Theorem}
	
	\newtheorem{lemma}[theorem]{Lemma}

\title{Valuations on Convex Functions}
\author{Andrea Colesanti,  Monika Ludwig and Fabian Mussnig}
\date{}

\begin{document}

\maketitle

\begin{abstract}
All continuous, $\sln$ and translation invariant valuations on the space of convex functions on $\R^n$ are completely classified.
\bigskip

\end{abstract}

A function $\oZ$ defined on a lattice $(\cL,\vee, \wedge)$ and taking values in an abelian semigroup is called a \emph{valuation} if
\begin{equation}\label{valuation}
\oZ(u\vee v)+\oZ(u\wedge v)=\oZ(u) +\oZ (v)
\end{equation}
for all $u,v\in \cL$. 
A function $\oZ$ defined on some subset $\cS$ of $\cL$ is called a valuation on $\cS$ if 
(\ref{valuation}) holds whenever $u,v, u\vee v, u\wedge v\in \cS$. For $\cS$ the set of compact convex sets, $\cK^n$, in $\R^n$ with $\vee$ denoting union and $\wedge$ intersection, valuations have been studied since  Dehn's solution of Hilbert's Third Problem in 1901 and interesting new ones keep arising (see, for example, \cite{HLYZ_acta}). The natural topology on $\cK^n$ is induced by the Hausdorff metric and continuous, $\sln$ and translation invariant valuations on $\cK^n$ were first classified by Blaschke.  The cele\-brated Hadwiger classification theorem establishes a complete classification of continuous, rigid motion invariant valuations on  $\cK^n$ and provides a characterization of intrinsic volumes.
See  \cite{Alesker99,Alesker01,Bernig:Fu, Haberl_sln,  Ludwig:Reitzner2,  HaberlParapatits_moments, Haberl:Parapatits_centro, LiMa, AbardiaWannerer} for some recent results on valuations on convex sets and  \cite{Hadwiger:V,Klain:Rota}  for information on the classical theory. 

More recently, valuations have been studied on function spaces. Here $\cS$ is a space of real valued functions and $u\vee v$ is the pointwise maximum of $u$ and $v$ while $u\wedge v$ is the pointwise minimum. For Sobolev spaces \cite{Ludwig:SobVal, Ludwig:Fisher, Ma2016} and $L^p$ spaces \cite{Tsang:Minkowski, Tsang:Lp, Ludwig:MM} complete classifications for valuations intertwining the $\sln$ were established. See also \cite{Ludwig:survey, Kone,Tuo_Wang_semi, Ober2014}. Moreover,  classical functionals for convex sets including the intrinsic volumes have been extended to the space of quasi-concave functions in  \cite{BobkovColesantiFragala} and  \cite{MilmanRotem2013} (see also \cite{ColesantiFragala, KlartagMilman2005}). A classification of rigid motion invariant valuations on quasi-concave functions is established in \cite{ColesantiLombardi}. For definable functions such a result was previously established in \cite{BaryshnikovGhristWright}.

The aim of this paper is to establish a complete classification of  $\sln$ and translation invariant valuations on convex functions. Let $\CV$ denote the space of convex functions $u: \R^n \to (-\infty, +\infty]$ which are {proper}, lower semicontinuous and coercive. Here a function is \emph{proper} if it is not identically $+\infty$ and it is \emph{coercive} if 
\begin{equation}\label{eq:cn_coercive}
\lim_{\lvert x\rvert \to+ \infty} u(x)=+\infty
\end{equation}
where $\lvert x\rvert$ is the Euclidean norm of $x$.
The space $\CV$ is one of the standard spaces in convex analysis and it is equipped with the topology associated to epi-convergence (see Section \ref{space}).

Let $n\ge 2$ throughout the paper.  A functional $\oZ:\CV\to\R$ is $\sln$ \!\emph{invariant} if 
$\oZ(u\circ \phi^{-1})=\oZ(u)$
for every $u\in\CV$ and $\phi\in\sln$.
It is \emph{translation invariant} if 
$\oZ(u\circ \tau^{-1})=\oZ(u)$ 
for every $u\in\CV$ and translation $\tau:\R^n\to \R^n$. In \cite{CavallinaColesanti}, a class of rigid motion invariant valuations on $\CV$ was introduced and classification results were established. However, the setting is different from our setting, as a different topology (coming from a notion of monotone convergence) is used in \cite{CavallinaColesanti} and monotonicity of the valuations is assumed. Variants of the functionals from \cite{CavallinaColesanti} also appear in our classification. We say that a functional $\oZ:\CV\to\R$ is \emph{continuous} if $\oZ(u)= \lim_{k\to \infty} \oZ(u_k)$ for every sequence $u_k\in\CV$ that epi-converges to $u\in\CV$.

\begin{Theorem}
A functional $\,\oZ :\CV\to [0,\infty)$ is a continuous, $\sln\!$ and translation invariant valuation if and only if there exist 
 a continuous function $\,\zeta_0: \R \to [0,\infty)$  and  a continuous function $\,\zeta_n: \R \to [0,\infty)$  with finite $(n-1)$-st moment  such that 
\begin{equation}\label{eq:the_non_neg_sln_inv_vals}
\oZ (u) = \zeta_0\big(\min\nolimits_{x\in\Rn}u(x)\big) + \int_{\dom u} \zeta_n\big(u(x)\big)\,dx
\end{equation}
for every $u\in\CV$. 
\end{Theorem}

\noindent
Here, a function $\zeta:\R \to [0,\infty)$ has finite $(n-1)$-st moment if $\int_0^{+\infty} t^{n-1} \zeta(t)\d t<+\infty$ and $\dom u$ is the domain of $u$, that is, $\dom u=\{x\in\R^n: u(x)<+\infty\}$. Since $u\in\CV$, the minimum of $u$ is attained and hence finite.

If the valuation in (\ref{eq:the_non_neg_sln_inv_vals}) is evaluated for a  (convex) indicator function $\Ind_K$ for $K\in\cK^n$ (where $\Ind_K(x)=0$ for $x\in K$ and $\Ind_K(x)= +\infty$ for $x\not\in K$), then 
$\zeta_0(0) V_0(K)  + \zeta_n(0)\, V_n(K)$ is obtained, where $V_0(K)$ is the Euler characteristic and $V_n(K)$ the $n$-dimensional volume of $K$. The proof of the theorem makes essential use of
 the following classification of continuous and $\sln$ invariant valuations on $\cP^n_0$, the space of convex polytopes which contain the origin. A functional $\oZ:\cP^n_0\to\R$ is a continuous and $\sln$ invariant valuation if and only if there are constants $c_0,c_n\in\R$ such that
\begin{equation}\label{thm:ludwig_reitzner}
\oZ(P)=c_0 V_0(P)+ c_n V_n(P)
\end{equation}
for every $P\in\cP^n_0$ (see, for example, \cite{LudwigReitzner_AP}).  For continuous and rotation invariant  valuations on $\cK^n$ that have polynomial behavior with respect to translations, a classification was established by Alesker \cite{Alesker99} but a classification of continuous and  rotation invariant valuations on $\cP_0^n$ is not known. 
It is also an open problem to establish a classification of continuous and rigid motion invariant valuations on $\CV$.

\section{The Space of Convex  Functions}\label{space}

We collect some properties of convex functions and of the space $\CV$. A basic reference is the book by Rockafellar \& Wets \cite{RockafellarWets} (see also \cite{dal_maso, Attouch}).
In particular, epi-convergence is discussed and  some properties of epi-convergent sequences of convex functions are established. For these results, conjugate functions are introduced. We also discuss piecewise affine functions and give a self-contained proof that they are dense in $\CV$. 

To every convex function $u:\Rn\to (-\infty,+\infty]$, there can be assigned several convex sets. For  $t\in(-\infty,+\infty]$, the \emph{sublevel sets}
\begin{equation*}
\{u<t\}=\{x\in\Rn:u(x)<t\},\quad \{u\leq t\}=\{x\in\Rn:u(x)\leq t\},
\end{equation*}
are convex. The domain, $\dom u$, of $u$ is convex and the \emph{epigraph} of $u$,
\begin{equation*}
\epi u =\{(x,y)\in\Rn\times\R: u(x)\leq y\},
\end{equation*}
is a convex subset of $\Rn\times\R$.

The lower semicontinuity of a convex function $u:\Rn\to(-\infty,+\infty]$ is equivalent to its epigraph being closed and to all sublevel sets, $\{u\leq t\}$, being closed. Such functions are also called \emph{closed}. The growth condition (\ref{eq:cn_coercive}) is equivalent to the boundedness of all sublevel sets $\{u\leq t\}$. Hence, $\{u\leq t\} \in \cK^n$ for $u\in\CV$ and $t\geq \min_{x\in\Rn} u(x)$.

For convex functions $u,v\in\CV$, the pointwise minimum $u\wedge v$ corresponds to the union of their epigraphs and therefore to the union of their sublevel sets. Similarly, the pointwise maximum $u\vee v$ corresponds to the intersection of the epigraphs and sublevel sets. Hence for all $t\in\R$
\begin{equation*}
\label{eq:lvl_sts_val}
\{u\wedge v \leq t\} = \{u\leq t\} \cup \{v\leq t\},\qquad \{u\vee v\leq t\}= \{u\leq t\} \cap \{v\leq t\},
\end{equation*}
where for $u\vee v\in\CV$ all sublevel sets are either empty or in $\cK^n$.
For $u\in\CV$, 
\begin{equation}\label{le:level_sets}\relint \{u\leq t\} \subseteq \{u<t\}\end{equation}for every $t>\min_{x\in\Rn} u(x)$, where $\relint$ is the relative interior (see \cite[Lemma 3.2]{CavallinaColesanti}).

\subsection{Epi-convergence}

A sequence $u_k: \Rn\to (-\infty, \infty]$ is \emph{epi-convergent} to $u:\Rn\to (-\infty, \infty]$ if for all $x\in\Rn$ the following conditions hold:
\begin{itemize}
	\item[(i)] For every sequence $x_k$ that converges to $x$,
			\begin{equation}\label{eq:gc_inf}
				u(x) \leq \liminf_{k\to \infty} u_k(x_k).
			\end{equation}
	\item[(ii)] There exists a sequence $x_k$ that converges to $x$ such that
			\begin{equation}\label{eq:gc_sup}
				u(x) = \lim_{k\to\infty} u_k(x_k).
			\end{equation}
\end{itemize}
In this case we also write $u=\elim_{k\to\infty} u_k$ and $u_k \eto u$.

\goodbreak

Equation (\ref{eq:gc_inf}) means, that $u$ is an asymptotic common lower bound for the sequence $u_k$. Consequently, (\ref{eq:gc_sup}) states that this bound is optimal.
The name epi-convergence is due to the fact, that this convergence is equivalent to the convergence of the corresponding epigraphs in the Painlev\'e-Kuratowski sense.
Another name for epi-convergence is  $\Gamma$-convergence (see \cite[Theorem 4.16]{dal_maso} and \cite[Proposition~7.2]{RockafellarWets}). We consider $\CV$ with the topology associated to epi-convergence.

\goodbreak
Immediately from the definition of epi-convergence we get the following result (see, for example, \cite[Proposition 6.1.]{dal_maso}).

\begin{lemma}
\label{le:g_conv_subseq}
If $u_k:\Rn\to(-\infty, \infty]$ is a sequence that epi-converges to $u:\Rn\to(-\infty, \infty]$, then also every subsequence $u_{k_i}$ of $u_k$ epi-converges to $u$.
\end{lemma}

For the following result, see, for example, \cite[Proposition 7.4 and Theorem 7.17]{RockafellarWets}.

\begin{lemma}\label{thm:g_comp}
If $u_k$ is a sequence of convex functions that epi-converges to a function $u$, then $u$ is convex and lower semicontinuous. Moreover, if $\dom u$ has non-empty interior, then $u_k$ converges uniformly to $u$ on every compact set that does not contain a boundary point of $\dom u$. 
\end{lemma}

We also require the following connection to pointwise convergence (see, for example, \cite[Example~5.13]{dal_maso}).

\begin{lemma}
\label{le:finite_g_p}
Let $u_k:\Rn\to\R$ be a sequence of finite convex functions and  $u:\Rn\to\R$ a finite convex function. Then $u_k$ is epi-convergent to $u$, if and only if $\,u_k$ converges pointwise to $u$.
\end{lemma}

\noindent
The last statement is no longer true if the functions may attain the value $+\infty$. In that case
\begin{equation*}
\elim_{k\to \infty} u_k(x) \leq \lim\nolimits_{k\to \infty} u_k(x),
\end{equation*}
for all $x\in\Rn$ such that both limits exist.

We want to connect epi-convergence of functions from $\CV$ with the convergence of their sublevel sets. The natural topology on $\cK^n$ is induced by the Hausdorff distance. For $K,L\subset \Rn$, we write 
$$K+L=\{x+y: x\in K, y\in L\}$$
for their \emph{Minkowski sum}. Let $B\subset\Rn$ be the closed, $n$-dimensional unit ball. For $K,L\in\cK^n$, the Hausdorff distance is
$$\delta(K,L)= \inf\{\e>0: K\subset L+ \e\,B, \, L\subset K+\e\,B\}.$$  
We write $K_i\to K$ as $i\to\infty$, if $\delta(K_i, K)\to 0$ as $i\to\infty$. For the next result we need the following description of Hausdorff convergence on $\cK^n$ (see, for example, \cite[Theorem 1.8.8]{Schneider:CB2}).

\begin{lemma}
\label{th:hd_conv}
The convergence $\lim_{i\to\infty} K_i = K$ in $\cK^n$ is equivalent to the following conditions taken together:
\begin{enumerate}
	\item[(i)] Each point in $K$ is the limit of a sequence $(x_i)_{i\in\N}$ with $x_i\in K_i$ for $i\in\N$.
	\item[(ii)] The limit of any convergent sequence $(x_{i_j})_{j\in\N}$ with $x_{i_j} \in K_{i_j}$ for $j\in\N$ belongs to $K$.
\end{enumerate}
\end{lemma}

Each sublevel set of a function from $\CV$ is either empty or in $\cK^n$. 
We say that $\{u_k \leq t\} \to \emptyset$ as $k\to\infty$ if there exists $k_0\in\N$ such that $\{u_k \leq t\} = \emptyset$ for  $k\geq k_0$. We include the proof of the following simple result, for which we did not find a suitable reference.

\begin{lemma}
\label{le:hd_conv_lvl_sets}
Let $u_k,u\in\CV$. If $u_k \eto u$ as $k\to\infty$, then $\{u_k\leq t\} {\to} \{u\leq t\}$ as $k\to\infty$ for every $t\in\R$ with $t\neq \min_{x\in\Rn} u(x)$.
\end{lemma}

\begin{proof}
First, let $t>\min_{x\in\Rn} u(x)$. 
For $x\in\relint\{u\leq t\}$,  it follows from (\ref{le:level_sets}) that $s=u(x)<t$. Since $u_k$ epi-converges to $u$, there exists a sequence $x_k$ that converges to $x$ such that $u_k(x_k)$ converges to $u(x)$. Therefore, there exist $\e>0$ and $k_0\in\N$ such that for all $k\geq k_0$
\begin{equation*}
u_k(x_k)\leq s+\e\leq t.
\end{equation*}
Thus, $x_k\in\{u_k \leq t\}$, which shows that $x$ is a limit of a sequence of points from $\{u_k\le t\}$. It is easy to see that this implies $(i)$ of Lemma \ref{th:hd_conv}.

Now, let $(x_{i_j})_{j\in\N}$ be a convergent sequence in $\{u_{i_j}\leq t\}$ with limit $x\in\Rn$. By Lemma~\ref{le:g_conv_subseq}, the subsequence $u_{i_j}$ epi-converges to  $u$. Therefore
\begin{equation*}
u(x)\leq \liminf_{j\to\infty} u_{i_j}(x_{i_j}) \leq t
\end{equation*}
which gives (ii) of Lemma \ref{th:hd_conv}.

Second, let $t<\min_{x\in\Rn} u(x)=\umin$. Since $\{u\leq t\}=\emptyset$,  we have to show that there exists $k_0\in\N$ such that for every $k\geq k_0$ and $x\in\R^n$,
\begin{equation*}
u_k(x) > t.
\end{equation*}
Assume that there does not exist such an index $k_0$. Then there are infinitely many points $x_{i_j}$ such that $u_{i_j}(x_{i_j}) \leq t$. Note, that
\begin{equation*}
x_{i_j} \in \{u_{i_j} \leq t\} \subseteq \{u_{i_j} \leq \umin+1\}.
\end{equation*}
By Lemma \ref{le:g_conv_subseq}, we know that $u_{i_j} \eto u$ and therefore we can apply the previous argument to obtain that $\{u_{i_j} \leq \umin+1\} \to \{u \leq \umin+1\}$, which shows that the sequence $x_{i_j}$ is bounded. Hence, there exists a convergent subsequence $x_{i_{j_k}}$ with limit $x\in\Rn$. Applying Lemma \ref{le:g_conv_subseq} again, we obtain that $u_{i_{j_k}}$ is epi-convergent to $u$ and therefore
\begin{equation*}
u(x)\leq \liminf_{k\to\infty} u_{i_{j_k}}(x_{i_{j_k}}) \leq t.
\end{equation*}
This is a contradiction. Hence $\{u_k\leq t\}$ must be empty eventually.
\end{proof}

\subsection{Conjugate functions and the cone property}

We require a uniform lower bound for an epi-convergent sequence of functions from $\CV$. This is established by showing that all epigraphs are contained in a suitable cone that is given by the function $a\,\vert x\vert +b$ with $a>0$ and $b\in\R$. To establish this \textit{uniform cone property} of an epi-convergent sequence, we use conjugate functions. 

For a convex function $u:\Rn\to(-\infty,+\infty]$, its \textit{conjugate function} $u^*:\Rn\to(-\infty,+\infty]$ is defined as
\begin{equation*}
u^*(y)=\sup_{x\in\Rn} \big( \<y,x\>-u(x)\big),\quad y\in\Rn.
\end{equation*}
Here $\<y,x\>$ is the inner product of $x,y\in\R^n$. If $u$ is a closed convex function, then also $u^*$ is a closed convex function and $u^{**}=u$. Conjugation reverses inequalities, that is, if $u\le v$, then $u^*\ge v^*$.

The \textit{infimal convolution} of two closed convex functions $u_1,u_2:\Rn\to(-\infty,+\infty]$ is defined by
\begin{equation*}
(u_1 \infconv u_2)(x)=\inf_{x=x_1+x_2} \big(u_1(x_1)+u_2(x_2)\big),\quad x\in\Rn.
\end{equation*}
This just corresponds to the Minkowski addition of the epigraphs of $u_1$ and $u_2$, that is
\begin{equation}
\label{eq:inf_conv_epi_add}
\epi (u_1 \infconv u_2) = \epi u_1 + \epi u_2. 
\end{equation}
We remark that
for two closed convex functions $u_1, u_2$ the infimal convolution $u_1 \infconv u_2$ need not be closed, even when it is convex.
If $u_1 \Box u_2 > -\infty$ pointwise, then
\begin{equation}
\label{eq:inf_conv_conj}
(u_1 \infconv u_2)^* = u_1^* + u_2^*.
\end{equation}
For $t>0$ and a closed convex function $u$ define the function $u_t$  by
\begin{equation*}
u_t(x)=t\, u\left(\tfrac{x}{t}\right).
\end{equation*}
For the convex conjugate of $u_t$, we have
\begin{equation}
\label{eq:epi_mult_conj}
u_t^*(y)=\sup_{x\in\Rn} \big( \<y,x\>-t\, u\left(\tfrac{x}{t}\right) \big)= \sup_{x\in\Rn} \big( \<y,t x\> - t\, u(x)\big) = t \,u^*(y)
\end{equation}
(see, for example,  \cite{Schneider:CB2}, Section 1.6.2).

\goodbreak
The next result shows a fundamental relationship between convex functions and their conjugates. It was first established by Wijsman (see \cite[Theorem 11.34]{RockafellarWets}).

\begin{lemma}\label{th:wijsman}
If $u_k,u\in\CV$, then
\begin{equation*}
u_k \eto u \quad \Longleftrightarrow \quad u_k^*\eto u^*.
\end{equation*}
\end{lemma}

The \emph{cone property} was established in \cite[Lemma 2.5]{ColesantiFragala} for functions in $\CV$.

\begin{lemma}\label{le:cone}
For $u\in\CV$, there exist constants $a,b \in \R$ with $a >0$ such that
\begin{equation*}
u(x)>a|x|+b
\end{equation*}
for every $x\in\Rn$.
\end{lemma}

Next, we extend this result to an epi-convergent sequence of functions in $\CV$ and obtain a \emph{uniform cone property}.

\begin{lemma}
\label{le:un_cone}
Let $u_k, u \in\CV$. If $u_k \eto u$, then there exist constants $a,b \in \R$ with $a >0$ such that
\begin{equation*}
u_k(x)>a\,\vert x\vert +b\,\, \text{ and }\,\ u(x)>a\,|x|+b 
\end{equation*}
for every $k\in\N$ and  $x\in\Rn$.
\end{lemma}

\begin{proof}
By Lemma \ref{le:cone}, there exist constants $c>0$ and $d$ such that
\begin{equation*}
u(x)> c \vert x\vert + d = l(x).
\end{equation*}
Switching to  conjugates gives
$u^* < l^*$.
Note that
\begin{equation*}
l^*(y)= \Big(\!\sup_{x\in\Rn} \big( \<y,x\> - c\, |x|\big)\Big) - d
\end{equation*}
and
\begin{equation*}
\sup_{x\in\Rn} \big(\<y,x\>-c\,|x| \big)=
\begin{cases} 0 &\mbox{if } |y| \leq c \\ 
+\infty & \mbox{if } |y| > c.
\end{cases}
\end{equation*}
Hence $l^* = \Ind_{cB} - d$, where $cB$ is the closed centered ball with radius $c$. Set $a=c/2>0$. Hence $aB$ is a compact subset of $\Int \dom u^*$. Therefore,  Lemma \ref{th:wijsman} and Lemma \ref{thm:g_comp} imply that $u_k^*$ converges uniformly to $u^*$ on $aB$. Since $u^*<-d$ on $aB$, there exists a constant $b$ such that
$u_k^*(y) < -b$
for every $y\in aB$ and $k\in\N$ and therefore
\begin{equation*}
u_k^* < \Ind_{aB} - b,
\end{equation*}
for every $k\in\N$. Consequently
\begin{equation*}
u_k(x) > a|x|+b
\end{equation*}
for every $k\in\N$ and  $x\in\Rn$.
\end{proof}

Note, that Lemma \ref{le:hd_conv_lvl_sets} and Lemma \ref{le:un_cone} are no longer true if $u\equiv {+\infty}$. For example, consider $u_k(x)=\Ind_{k\,r( B+k^2 x_0)}$ for some $r>0$ and $x_0 \in \Rn\backslash\{0\}$. Then $u_k$ epi-converges to $u$ but every set $\{u_k \leq t\}$ for $t\geq 0$ is a ball of radius $kr$. In this case, the sublevel sets are not even bounded. Moreover, it is clear that there does not exist a uniform pointed cone that contains all the sets $\epi u_k$.

\subsection{Piecewise affine functions}

A polyhedron is the intersection of finitely many closed halfspaces. A function $u\in\CV$ is called \textit{piecewise affine}, if there exist finitely many $n$-dimensional convex polyhedra $C_1,\ldots,C_m$ with pairwise disjoint interiors such that $\bigcup_{i=1}^m C_i =\Rn$ and the restriction of $u$ to each $C_i$ is an affine function. The set of piecewise affine and convex functions will be denoted by $\CVpa$. We call $u\in\CV$ a \textit{finite element} of $\CV$ if $u(x)<+\infty$ for every $x\in\R$. Note that piecewise affine and convex functions are finite elements of $\CV$.
We want to show that $\CVpa$ is dense in $\CV$ and use the Moreau-Yosida approximation in our proof. That $\CVpa$ is dense in $\CV$ can also be deduced from more general results  (see,  \cite[Corollary 3.42] {Attouch}).

\goodbreak
Let $u\in\CV$ and $t>0$. Set  $q(x)=\tfrac 12 |x|^2$ and  recall that $q_t(x)=\frac{t}2\, q(\frac xt)$. The \textit{Moreau-Yosida approximation} of $u$ is defined as
\begin{equation*}
e_t u =u \infconv q_t,
\end{equation*}
or equivalently
\begin{equation*}
e_t u(x) = \inf_{y\in\Rn} \big( u(y)+\tfrac{1}{2t} |x-y|^2\big) = \inf_{x_1+x_2=x} \big(u(x_1)+\tfrac{1}{2t} |x_2|^2\big).
\end{equation*}
See, for example, \cite[Chapter 1, Section G]{RockafellarWets}. We require the following simple properties of the Moreau-Yosida approximation.

\begin{lemma}
\label{le:mo_yo_is_nice}
For $u\in\CV$, the Moreau-Yosida approximation $e_t u$ is a finite element of $\CV$ for every $t>0$. Moreover, $e_t u(x) \leq u(x)$ for  $x\in\Rn$ and $t>0$.
\end{lemma}
\begin{proof}
Fix $t >0$. Since
\begin{equation*}
\inf_{x_1+x_2=x} \big(u(x_1)+\tfrac{1}{2t}|x_2|^2\big) \leq u(x)+\tfrac{1}{2t}|0|^2
\end{equation*}
for $x\in\R^n$, we have $e_t u(x) \leq u(x)$ for all $x\in\Rn$. Since $u$ is proper, there exists $x_0 \in\Rn$ such that $u(x_0)<+\infty$. This shows that
\begin{equation*}
e_t u(x) = \inf_{x_1+x_2 = x} \big(u(x_1) + \tfrac{1}{2t} |x_2|^2\big) \leq u(x_0) + \tfrac{1}{2t} |x-x_0|^2 < +\infty
\end{equation*}
for $x\in\R^n$,  which shows that $e_t u$ is finite. Using (\ref{eq:inf_conv_epi_add}) we obtain that
\begin{equation*}
\epi e_t u = \epi u + \epi q_t.
\end{equation*}
It is therefore easy to see, that $e_t u$ is a convex function such that $\lim_{|x|\to +\infty} e_t u(x) = +\infty$. \end{proof}

\begin{lemma}
\label{le:mo_yo_epi_conv}
For every $u\in\CV$, 
$\,\elim_{t \to 0^+} e_t u =u$.

\end{lemma}
\begin{proof}
By Lemma \ref{th:wijsman}, we have $e_t u \eto u$ if and only if $(e_t u)^* \eto u^*$. By the definition of $e_t$, (\ref{eq:inf_conv_conj})  and (\ref{eq:epi_mult_conj}), we have
\begin{equation*}
(e_t u)^* = (u\infconv q_t)^* = u^* + t q^*.
\end{equation*}
Therefore, we need to show that $u^* + t q^* \eto u^*$. For $q(x)=\tfrac 12 |x|^2$, we have $q=q^*$. 
Since epi-convergence is equivalent to pointwise convergent if the functions are finite, it follows  that $\elim_{t \to 0^+} t q^* = {0}$. 
It is now easy to see that $\elim_{t\to 0^+} \big(u^* + t q^* )= u^*$ and therefore $\elim_{t\to 0^+}(e_t u)^* = u^*$.
\end{proof}

\goodbreak

\begin{lemma}
\label{le:pa_dense}
$\CVpa$ is dense in $\CV$.
\end{lemma}
\begin{proof}
By Lemma \ref{le:finite_g_p}, epi-convergence coincides with pointwise convergence on finite functions in $\CV$. Therefore, it is easy to see that $\CVpa$ is epi-dense in the finite elements of $\CV$. Now for arbitrary $u\in\CV$ it follows from Lemma \ref{le:mo_yo_is_nice} that $e_t u$ is a finite element of $\CV$. Since Lemma \ref{le:mo_yo_epi_conv} gives that $\elim_{t\to 0^+}e_t u =u$, the finite elements of $\CV$ are a dense subset of $\CV$. Since denseness is transitive, the piecewise affine functions are a dense subset of $\CV$.  
\end{proof}

\section{Valuations on Convex Functions}

The functionals that appear in the theorem are discussed.  It is shown that they are continuous, $\sln$ and translation invariant valuations on $\CV$.

\begin{lemma}
\label{le:min_is_a_val}
For $\zeta\in C(\R)$, the map
\begin{equation}
\label{eq:generalized_euler_char}
u\mapsto \zeta\big(\min\nolimits_{x\in\Rn} u(x)\big)
\end{equation}
is a continuous, $\sln$ and translation invariant valuation on $\CV$.
\end{lemma}
\begin{proof}
Let $u\in\CV$. Since
\begin{equation*}
\min_{x\in\Rn} u(x) = \min_{x\in\Rn} u(\tau x) = \min_{x\in\Rn} u(\phi^{-1} x),
\end{equation*}
for every $\phi\in\sln$ and translation $\tau:\Rn\to\Rn$,  (\ref{eq:generalized_euler_char}) defines an $\sln$ and translation invariant map. If $u,v\in\CV$ are such that $u\wedge v\in\CV$, then clearly
\begin{equation*}
\min_{x\in\Rn} (u\wedge v)(x) = \min \{\min_{x\in\Rn} u(x), \min_{x\in\Rn} v(x) \}.
\end{equation*}
By \cite[Lemma 3.7]{CavallinaColesanti} we have
\begin{equation*}
\min_{x\in\Rn} (u\vee v)(x) = \max \{\min_{x\in\Rn} u(x), \min_{x\in\Rn} v(x) \}.
\end{equation*}
Hence, a function $\zeta\in C(\R)$ composed with the minimum of a function $u\in\CV$ defines a valuation on $\CV$. The continuity of \eqnref{eq:generalized_euler_char} follows from Lemma~\ref{le:hd_conv_lvl_sets}.
\end{proof}

Let $\zeta\in C(\R)$ be non-negative. For $u\in \CV$, define
\begin{equation*}
\oZ_\zeta(u)=\int_{\dom u} \zeta(u(x)) \d x.
\end{equation*}
We want to investigate conditions on $\zeta$ such that $\oZ_\zeta$ defines a continuous valuation on $\CV$.

It is easy to see, that in order for $\oZ_\zeta(u)$ to be finite for every $u\in\CV$, it is necessary for $\zeta$ to have {finite $(n-1)$-st moment}. 
Indeed, if $u(x)=|x|$, then
\begin{equation}
\label{eq:finite_moment_necessary}
\oZ_\zeta(u)=\int_{\Rn} \zeta(|x|)\d x = n\,v_n \int_0^{+\infty} t^{n-1} \zeta(t) \d t,
\end{equation}
where $v_n$ is the volume of the $n$-dimensional unit ball. We will see in Lemma \ref{le:mu_f_moment}, that this condition is also sufficient. For this, we require the following result.

\begin{lemma}
\label{le:val_small_on_complement_of_lvl_sets}
Let $u_k$ be a sequence in $\CV$ with epi-limit $u\in\CV$. If $\zeta\in C(\R)$ is non-negative with finite $(n-1)$-st moment, then, for every $\e>0$, there exist $t_0\in\R$ and $k_0\in\N$ such that
\begin{equation*}
\int_{\dom u \cap \{u>t\}} \zeta(u(x))\d x < \e \,\,\,\text{ and }\,\,\, \int_{\dom u_k\cap \{u_k >t\}} \zeta(u_k(x)) \d x < \e
\end{equation*}
for every  $t\geq t_0$ and $k\geq k_0$.
\end{lemma}
\begin{proof}
Without loss of generality, let $\min_{x\in\Rn}u(x)=u(0)$. By the definition of epi-con\-ver\-gence, there exists a sequence $x_k$ in $\Rn$ such that $x_k\to 0$ and $u_k(x_k)\to u(0)$. Therefore, there exists $k_0\in\N$ such that $|x_k|<1$ and $u_k(x_k) < u(0)+1$ for every $k\geq k_0$. By Lemma \ref{le:un_cone}, there exist constants $a>0$ and $\bar{b}\in\R$, such that
\begin{equation*}
u(x),u_k(x) > a|x|+\bar{b},
\end{equation*}
for every $x\in\Rn$ and $k\in\N$. Setting $\widetilde{u}_k(x)=u_k(x-x_k)$, we have
\begin{equation*}
\widetilde{u}_k(x)>a|x-x_k|+\bar{b}\geq a |x|-a|x_k|+\bar{b} \geq a |x|+ (\bar{b}-a),
\end{equation*}
for every $k\geq k_0$. Hence, with $b=\bar{b}-a$, we have
\begin{equation}\label{uutildelower}
u(x),\widetilde{u}_k(x)>a|x|+b,
\end{equation}
for every $x\in\Rn$ and $k\geq k_0$.

We write $x=r\omega$ with $r\in[0,+\infty)$ and $\omega \in\Sph$.  For $u(r\omega)\ge 1$, we obtain from (\ref{uutildelower}) that
\begin{equation}\label{absch}
r^{n-1} < \left( \frac{u(r\omega)}{a}-\frac{b}{a}\right)^{n-1} \le c\, u(r\omega)^{n-1},\quad 
r^{n-1} < c\, \widetilde{u}_k(r\omega)^{n-1},
\end{equation}
for every $r\in[0,+\infty)$, $\omega \in \Sph$ and $k\geq k_0$, where $c$ only depends on $a,b$ and the dimension $n$. Now choose $\bar{t}_0 \geq \max\{1,2(u(0)+1)-b\}$. Then for all $t\geq \bar{ t}_0$ 
\begin{equation}\label{half}
\frac{t-u(0)}{t-b} \geq \frac{1}{2},\quad \frac{t-(u(0)+1)}{t-b}\geq \frac{1}{2}.
\end{equation}

For $\omega \in \Sph$, let $v_\omega(r)=u(r\omega)$. The function $v_\omega$ is non-decreasing  and convex  on $[0,+\infty)$. So, in particular, the left and right derivatives, $v'_{\omega,l}, v'_{\omega,r}$ of $v_\omega$ exist and for the
subgradient $\partial v_\omega (r) = [v'_{\omega,l}, v'_{\omega,r}]$, it follows from $r<\bar r$ that $\eta \le \bar \eta$ for $\eta\in\partial v_\omega (r)$ and $\bar \eta \in \partial v_\omega (\bar r)$. 

For $t\geq \bar{t}_0$, set
\begin{equation*}
D_\omega(t)=\{r\in[0,+\infty)\,:\, t< u(r\omega)<+\infty\}.
\end{equation*}
For every $\omega\in\Sph$, the set $D_\omega(t)$ is either empty or there exists 
\begin{equation}\label{absch2}
r_\omega(t)=\inf D_\omega(t) \leq \tfrac{t-b}{a}
\end{equation}
and $v_\omega(r_\omega(t))=t$. Therefore, if $D_\omega(t)$ is non-empty, we have 
\begin{equation*}
t-u(0) \le \xi \,r_\omega(t)
\end{equation*}
for $\xi\in \partial v_\omega(r_\omega(t))$. Hence, it follows from (\ref{absch2}) and (\ref{half}) that
\begin{equation}\label{lower}
\vartheta \geq  \xi \geq \frac{t-u(0)}{r_\omega(t)} \geq \frac{a(t-u(0))}{t-b} \geq \frac{a}{2},
\end{equation}
for all $r\in D_\omega(t)$, $\vartheta\in \partial v_\omega(r)$ and $\xi \in \partial v_\omega(r_\omega(t))$. Similarly, setting $\widetilde{v}_{k,\omega}(r)= \widetilde{u}_k(r\omega)$ and
\begin{equation*}
\widetilde{D}_{k,\omega}(t)=\{r\in[0,+\infty)\,:\, t< u_k(r\omega)<+\infty\},
\end{equation*}
it is easy to see that $\widetilde{v}_{k,\omega}$ is convex on $[0,+\infty)$ and monotone increasing on $\widetilde{D}_{k,\omega}(t)$ for all $k\geq k_0$. By the choice of $\bar{t}_0$ and (\ref{half}), for $t\ge \bar{t}_0$
\begin{equation*}
\vartheta \geq \frac{a}{2},
\end{equation*}
for all $r\in D_{k,\omega}(t)$, $k\geq k_0$ and $\vartheta \in\partial \widetilde{v}_{k,\omega}(r)$. Recall, that as a convex function $v_\omega$ is locally Lipschitz and differentiable almost everywhere on the interior of its domain. Using polar coordinates, (\ref{absch}) and the substitution $v_\omega(r)=s$, we obtain from (\ref{lower}) that
\begin{equation}\label{eq:bar_c}
\begin{array}{rcl}
\displaystyle
\int_{\dom u\cap \{u>t\}} \zeta(u(x))\d x &=& \displaystyle  \int_{\Sph} \int_{D_\omega(t)} r^{n-1} \zeta(v_\omega(r)) \d r \d \omega \\[12pt]
&\leq& \displaystyle c \int_{\Sph} \int_{D_\omega(t)} v_\omega(r)^{n-1} \zeta(v_\omega(r)) \d r \d \omega \\[12pt]
&\leq& \displaystyle  \frac{2 \,n\,v_n\,c}{a} \int_{t}^{+\infty} s^{n-1} \zeta(s) \d s
\end{array}
\end{equation}
for every $t\geq \bar{t}_0$. In the same way,
\begin{multline*}
\int_{\dom u_k \cap \{u_k >t\}} \zeta(u_k(x)) \d x = \int_{\dom \widetilde{u}_k \cap \{\widetilde{u}_k >t\}} \zeta(\widetilde{u}_k(x)) \d x
\leq  \frac{2 \,n\,v_n\,c}{a}  \int_{t}^{+\infty} s^{n-1} \zeta(s) \d s,
\end{multline*}
for every $t\geq \bar{t}_0$ and $k\geq k_0$ with the same constant ${c}$ as in (\ref{eq:bar_c}). The statement now follows, since $\zeta$ is non-negative and has finite $(n-1)$-st moment.
\end{proof}

\goodbreak

\begin{lemma}
\label{le:mu_f_moment}
Let $\zeta\in C(\R)$ be non-negative. Then $\oZ_\zeta(u)<+\infty$ for every $u\in\CV$ if and only if $\zeta$ has finite $(n-1)$-st moment.
\end{lemma}
\begin{proof}
As already pointed out in (\ref{eq:finite_moment_necessary}), it is necessary for $\zeta$ to have finite $(n-1)$-st moment in order for $\oZ_\zeta$ to be finite.

Now let $u\in\CV$ be arbitrary, let $\zeta$ have finite $(n-1)$-st moment and let $\umin=\min_{x\in\Rn}u(x)$. By Lemma \ref{le:val_small_on_complement_of_lvl_sets}, there exists $t\in\R$ such that
\begin{equation*}
\int_{\dom u \cap \{u > t\}} \zeta(u(x))\d x \leq 1.
\end{equation*}
It follows that
\begin{align*}
\oZ_\zeta(u) &= \int_{\dom u} \zeta(u(x)) \d x\\
&= \int_{\{u \leq t\}} \zeta(u(x))\d x + \int_{\dom u \cap \{u > t\}} \zeta(u(x))\d x\\[3pt]
&\leq \max\nolimits_{s\in[ \umin,t]} \zeta(s) \,V_n(\{u \leq t\}) + 1
\end{align*}
and hence  $\oZ_\zeta(u)<\infty$.
\end{proof}

\begin{lemma}
\label{le:mu_f_is_cont}
For $\zeta\in C(\R)$ non-negative and with finite $(n-1)$-st moment, the functional $\oZ_\zeta$ is continuous on $\CV$.
\end{lemma}
\begin{proof}
Let $u\in\CV$ and let $u_k$ be a sequence in $\CV$ such that $u_k\eto u$. Set $\umin=\min_{x\in\Rn}u(x)$.
By Lemma \ref{le:val_small_on_complement_of_lvl_sets}, it is enough to show that
\begin{equation*}
\int_{\{u_k \leq t\}} \zeta(u_k(x))\d x \to \int_{\{u \leq t\}} \zeta(u(x))\d x
\end{equation*}
for every fixed $t>\umin$. 
Lemma \ref{le:hd_conv_lvl_sets} implies  that $\{u_k \leq t\} \to \{u\leq t\}$ in the Hausdorff metric. By Lemma \ref{le:un_cone}, there exists $b\in\R$ such that $u(x),u_k(x) > b$ for $x\in\Rn$ and $k\in\N$. 
Set $c=\max_{s\in [b,t]} \zeta(s)\geq 0$.
We  distinguish the following cases. 

First, let $\dim (\dom u) <n$. 	In this case $V_n(\{u\leq t\})=0$ and since volume is continuous on convex sets, \mbox{$V_n(\{u_k \leq t\})\to 0$}. Hence,
	\begin{equation*}
	0 \leq  \int_{\{u_k\leq t\}} \zeta(u_k(x))\d x \leq c\,V_n(\{u_k \leq t\})  \to 0.
	\end{equation*}
	
\goodbreak
Second, let  $\dim (\dom u) = n$.
	In this case, $\{u\leq t\}$ is a set in $\cK^n$ with non-empty interior. Therefore, for $\e>0$ there exist $k_0\in\N$ and $C\in\cK^n$ such that for every $k\geq k_0$ the following hold:
	\begin{equation*}
	C \subset \Int(\{u\leq t\})\cap \{u_k \leq t\},
	\end{equation*}
	\begin{equation*}
	V_n(\{u\leq t\} \cap C^c) \leq \frac{\e}{3 c},
	\end{equation*}
	\begin{equation*}
	V_n(\{u_k\leq t\} \cap C^c) \leq \frac{\e}{3 c},
	\end{equation*}
where $C^c$ is the complement of $C$.
	Note, that $u(x), u_k(x) \in [b,t]$ for  $x\in C$ and $k\geq k_0$. Since $C\subset \Int \dom u$, Lemma \ref{thm:g_comp} implies  that $u_k$ converges to $u$ uniformly on $C$. Since $\zeta$ is continuous, the restriction of $\zeta$ to $[b,t]$ is uniformly continuous. Hence, $\zeta\circ u_k$ converges uniformly to $\zeta\circ u$ on $C$. Therefore, there exists $k_1\geq k_0$ such that
	\begin{equation*}
	|\zeta(u(x))-\zeta(u_k(x))| \leq \frac{\e}{3 V_n(C)},
	\end{equation*}
for all $x\in C$ and $k\geq k_1$. 
This gives
	\begin{align*}
	&\Big|\int_{\{u\leq t\}}  \zeta(u(x)) \d x - \int_{\{u_k \leq t\}} \zeta(u_k(x)) \d x \Big|\\
	&\leq \int_C |\zeta(u(x)) - \zeta(u_k(x))| \d x + \int_{\{u\leq t\} \cap C^c} \zeta(u(x)) \d x + \int_{\{u_k \leq t\} \cap C^c} \zeta(u_k(x)) \d x \\
	&\leq V_n(C) \frac{\e}{3V_n(C)} + c \frac{\e}{3c} +  c \frac{\e}{3c} = \e,
	\end{align*}
	for  $k\geq k_1$. The statement now follows, since $\e>0$ was arbitrary.
\end{proof}

\begin{lemma}
\label{le:muf_is_a_val}
For $\zeta\in C(\R)$ non-negative and with finite $(n-1)$-st moment,  the functional $\oZ_\zeta$ is an  $\sln$ and translation invariant valuation on $\CV$.
\end{lemma}
\begin{proof}
It is easy to see that $\oZ_\zeta$ is $\sln$ and translation invariant. It remains to show the valuation property. Let $u,v\in\CV$  be such that $u\wedge v \in\CV$. We have
\begin{equation*}
\oZ_\zeta(u\wedge v) = \int_{\dom  v \cap \{v<u\}} \zeta(v(x)) \d x + \int_{\dom v \cap \{u=v\}} \zeta(v(x)) \d x + \int_{\dom u \cap \{u<v\} } \zeta(u(x)) \d x,
\end{equation*}
\begin{equation*}
\oZ_\zeta(u\vee v) = \int_{ \dom u \cap \{v<u\}} \zeta(u(x)) \d x + \int_{\dom u \cap \{u=v\}} \zeta(u(x)) \d x + \int_{\dom v \cap \{u < v\}} \zeta(v(x)) \d x.
\end{equation*}
Hence,
\begin{equation*}
\oZ_\zeta(u\wedge v ) + \oZ_\zeta(u\vee v) = \oZ_\zeta(u)+\oZ_\zeta(v)
\end{equation*}
and the valuation property is proved. \end{proof}

\section{Valuations on Cone and Indicator Functions}

Let $\cK^n_0$ be the set of compact convex sets which contain the origin.
For $K\in\cK^n_0$, we define the convex function $\l_K:\R^n\to [0,\infty]$ via
\begin{equation*}
\epi \l_K = \pos (K\times\{1\}),
\end{equation*}
where $\pos$ denotes the positive hull. This means that the epigraph of $\l_K$ is a cone with apex at the origin and $\{\l_K\leq t \}=t \, K$ for all $t \geq 0$. It is easy to see that $\l_K$ is an element of $\CV$ for all $K\in\cK^n_0$. We have $\dom \l_K=\Rn$ if and only if $K$ contains the origin in its interior. %Let $\cP^n_0$ denote the set of convex polytopes in $\R^n$ that contain the origin. 
If $P\in\cP^n_0$ contains the origin in its interior,  then $\l_P\in\CVpa$. 
For $K\in\cK^n_0$ and $t\in\R$, we call the function $\ell_K +t$ a \textit{cone function} and we call the function $\Ind_K+t$ an \textit{indicator function}. Cone and indicator functions play a special role in our proof.

The next result shows that to classify continuous and translation invariant valuations on $\CV$, it is enough to know the behavior of these valuations on cone functions.
The main argument of the following lemma  is due to \cite[Lemma 8]{Ludwig:SobVal}, where it was used for functions on Sobolev spaces.

\begin{lemma}
\label{le:reduction}
Let $\<A,+\>$ be a topological abelian semigroup with cancellation law and let $\,\oZ_1, \oZ_2:\CV\to \<A,+\>$ be continuous, translation invariant valuations. If $\,\oZ_1(\l_P+t)=\oZ_2(\l_P+t)$ for every $P\in\cP_0^n$ and $t\in\R$, then $\oZ_1 \equiv \oZ_2$ on $\CV$.
\end{lemma}
\begin{proof}
By Lemma \ref{le:pa_dense} and the continuity of $\oZ_1$ and $\oZ_2$, it suffices to show that $\oZ_1$ and $\oZ_2$ coincide on $\CVpa$. So let $u\in\CVpa$ and set $U=\epi u$. Note, that $U$ is a convex polyhedron in $\R^{n+1}$ and that none of the facet hyperplanes of $U$ is parallel to the $x_{n+1}$-axis. Here, we say that a hyperplane $H$ in $\R^{n+1}$ is a \textit{facet hyperplane} of $U$ if its intersection with the boundary of $U$ has positive $n$-dimensional Hausdorff measure. Furthermore, we call $U$ \textit{singular} if $U$ has $n$ facet hyperplanes whose intersection contains a line parallel to $\{x_{n+1}=0\}$. Since $\oZ_1$ and $\oZ_2$ are continuous, we can assume that $U$ is not singular.\par
Since $U$ is not singular and $u\in\CVpa$, there exists a unique vertex, $\bar p$ of $U$ with smallest $x_{n+1}$ coordinate. We use induction on the number $m$ of facet hyper\-planes of $U$ that are not passing through $\bar{p}$. If $m=0$, then there exist $P\in\cP_0^n$ and $t\in\R$ such that $u$ is a translate of $\l_P+t$. Since $\oZ_1$ and $\oZ_2$ are translation invariant, it follows that $\oZ_1(u) = \oZ_2(u)$.

Now let $U$ have $m>0$ facet hyperplanes that are not passing through $\bar{p}$ and assume that $\oZ_1$ and $\oZ_2$ coincide for all functions with at most $(m-1)$ such facet hyperplanes. 
Let $p_0=(x_0,u(x_0))\in\R^{n+1}$ where $x_0\in\Rn$ is a vertex of $U$ with maximal $x_{n+1}$-coordinate and let $H_1,\ldots,H_j$ be the facet hyperplanes of $U$ through $p_0$ such that the corresponding facets of $U$ have infinite $n$-dimensional volume. Note, that $H_1,\ldots,H_j$ do not contain $\bar{p}$ and therefore there is at least one such hyperplane. Define $\bar{U}$ as the polyhedron bounded by the intersection of all facet hyperplanes of $U$ with the exception of $H_1,\ldots,H_j$. Since $U$ is not singular, there exists a function $\bar{u}\in\CVpa$ with $\dom \bar{u}=\Rn$ such that $\bar{U}=\epi \bar{u}$. Note, that $\bar{U}$ has at most $(m-1)$ facet hyperplanes not containing $\bar{p}$. Hence, by the induction hypothesis 
$$\oZ_1(\bar{u})=\oZ_2(\bar{u}).$$ 
Let $\overline{H}_1,\ldots,\overline{H}_i$ be the facet hyperplanes of $\bar{U}$ that contain $p_0$ such that the corresponding facets of $\bar{U}$ have infinite $n$-dimensional volume. Choose suitable hyperplanes $\overline{H}_{i+1},\ldots,\overline{H}_k$ not parallel to the $x_{n+1}$-axis and containing $p_0$ so that the hyperplanes $\overline{H}_1,\ldots,\overline{H}_k$ bound a polyhedral cone with apex $p_0$ that is contained in $\bar U$, has $\overline{H}_1,\ldots,\overline{H}_i$ among its facet hyperplanes and contains $\{x_0\} \times [u(x_0),+\infty)$. Define $\l$ as the piecewise affine function determined by this polyhedral cone. Notice, that $\l$ is a translate of $\l_P+u(x_0)$, where $P\in\cP_0^n$ is the projection onto the first $n$ coordinates of the intersection of the polyhedral cone with $\{x_{n+1}=u(x_0)+1\}$.
Hence, $\oZ_1$ and $\oZ_2$ coincide on $\l$. Set $\bar{\l}=u\vee \l$. The epigraph of $\bar{\l}$ is again a polyhedral cone with apex $p_0$. Hence $\bar{\l}$ is a translate of $\l_{\bar{P}}+u(x_0)$ with $\bar{P}\in\cP_0^n$ since it is bounded by hyperplanes containing $p_0$ that are not parallel to the $x_{n+1}$-axis. Therefore, $\oZ_1$ and $\oZ_2$ also coincide on $\bar{\l}$. We now have
\begin{equation*}
u\wedge \l = \bar{u}, \quad u \vee \l = \bar{\l}.
\end{equation*}
From the valuation property of $\oZ_i$, $i=1,2$, we obtain
\begin{equation*}
\oZ_1(u)+\oZ_1(\l) = \oZ_1(\bar{u})+\oZ_1(\bar{\l}) = \oZ_2(\bar{u})+\oZ_2(\bar{\l})= \oZ_2(u)+\oZ_2(\l),
\end{equation*}
which completes the proof.
\end{proof}

\goodbreak

Next, we study the behavior of a continuous and $\sln$ invariant valuation on cone and indicator functions.

\begin{lemma}
\label{le:sln_grwth}
If $\,\oZ:\CV\to\R$ is a continuous and $\sln$\! invariant valuation, then there exist continuous functions $\psi_0,\psi_n,\zeta_0,\zeta_n:\R\to\R$ such that
\begin{eqnarray*}
\oZ(\l_P+t)&=&\psi_0(t)+\psi_n(t) V_n(P),\\[3pt]
\oZ(\Ind_P+t)&=&\zeta_0(t)+\zeta_n(t) V_n(P)
\end{eqnarray*}
for every $P\in\cP_0^n$ and $t\in\R$.
\end{lemma}
\begin{proof}
For $t\in\R$, define $\oZ_t:\cP_0^n\to\R$ as
\begin{equation*}
\oZ_t(P)=\oZ(\l_P+t).
\end{equation*}
It is easy to see that $\oZ_t$ defines a continuous, $\sln$ invariant valuation on $\cP_0^n$ for every $t\in\R$. Therefore, by (\ref{thm:ludwig_reitzner}), for every $t\in\R$ there exist constants $c_{0,t},c_{n,t}\in\R$ such that
\begin{equation*}
\oZ(\l_P+t)=\oZ_t(P)=c_{0,t}+c_{n,t}V_n(P),
\end{equation*}
for every $P\in\cP_0^n$. This defines two functions $\psi_0(t)=c_{0,t}$ and $\psi_n(t)=c_{n,t}$. Taking $P\in\cP_0^n$ with $\dim P < n$, we have $V_n(P)=0$. By the continuity of $\oZ$,
\begin{equation*}
t\mapsto \oZ(\l_P+t) = \psi_0(t)
\end{equation*}
is continuous, which implies that $\psi_0$ is a continuous function. Similarly, taking $Q\in\cP_0^n$ with $V_n(Q)>0$, we see that
\begin{equation*}
t\mapsto \psi_n(t)=\frac{\oZ(\l_Q+t)-\psi_0(t)}{V_n(Q)},
\end{equation*}
can be expressed as the difference of two continuous functions and is therefore continuous itself.
Using $P\mapsto \oZ(\Ind_P+t)$ we get the  corresponding results for the functions $\zeta_0$ and $\zeta_n$.
\end{proof}

For a continuous and $\sln$ invariant valuation $\oZ:\CV\to\R$, we call the functions $\psi_0$ and $\psi_n$ from Lemma \ref{le:sln_grwth} the \textit{cone growth functions} of $\oZ$. The functions $\zeta_0$ and $\zeta_n$ are its \textit{indicator growth functions}. By Lemma \ref{le:reduction}, we immediately get the following result.

\begin{lemma}
\label{le:sln_val_determined_by_cone_grwth_functions}
Every continuous, $\sln$ and translation invariant valuation $\oZ:\CV\to\R$ is uniquely determined by its cone growth functions.
\end{lemma}

In order to classify valuations, we want to determine how the cone growth functions and the indicator growth functions are related.

\begin{lemma}
\label{le:diff_grwth_fncts}
For $k\ge 1$, let $\oZ:\C{k}\to\R$ be a continuous, translation invariant valuation and let $\psi\in C(\R)$. If
\begin{equation}\label{ind_simple}
\oZ(\l_P+t) = \psi(t) V_k(P)
\end{equation}
for every $P\in\cP_0^k$ and $t\in\R$, then
\begin{equation*}
\oZ(\Ind_{[0,1]^k}+t) = \frac{(-1)^k}{k!} \frac{\d^k}{\d t^k} \psi(t)
\end{equation*}
for every $t\in\R$. In particular, $\psi$ is $k$-times differentiable.
\end{lemma}

\begin{proof}
To explain the idea of the proof, we first consider the case $k=1$. For $h>0$, let $u_h=\l_{[0,1/h]}$, that is,
$u^h(x)=+\infty$ for $x<0$ and $u^h(x) =h\, x$ for $x\ge0$. Define $v^h:\R\to [0,+\infty]$ by $v^h = u^h+ \Ind_{[0,1]}$. Since $\oZ$ is a translation invariant valuation and by (\ref{ind_simple}), we obtain 
$$\oZ(v^h+t)= \oZ(u^h+t)- \oZ(u^h+h +t)=\frac1 h\Big(\psi(t) - \psi(t+h)\Big)$$
for $t\in\R$. As $h\to 0$, the epi-limit of $v^h+t$ is $\Ind_{[0,1]}+t$. Since $\oZ$ is continuous, we thus obtain
$$\oZ(\Ind_{[0,1]}+t) = \lim_{h\to 0^+} \frac1 h\Big(\psi(t) - \psi(t+h)\Big)$$
for $t\in\R$. Hence $\psi$ is differentiable from the right at every $t\in\R$. Since $v^h+t-h \eto \Ind_{[0,1]}+t$ as $h\to 0$, we also obtain
$$\oZ(\Ind_{[0,1]}+t) = \lim_{h\to 0^+}\big(  \oZ(u^h+t-h)- \oZ(u^h+t)\big) =  \lim_{h\to 0^+} \frac1 {h}\Big(\psi(t- h) - \psi(t)\Big). $$
Hence $\psi$ is also differentiable from the left at every $t\in\R$ and $\oZ(\Ind_{[0,1]}+t)=- \psi'(t)$. This concludes the proof for $k=1$.
\goodbreak

Next, let $\{e_1,\ldots,e_k\}$ denote the standard basis of $\R^k$ and set $e_0=0$. For $h=(h_1,\ldots,h_k)$ with $0< h_1 \leq \cdots \leq h_k$ and $0\leq i < k$, define the function $u_i^h$  through its sublevel sets as
\begin{equation*}
\{u_i^h<0\}=\emptyset,\quad \{u_i^h \leq s\} = [0,e_0]+ \cdots + [0,e_i] + \conv\{0,s\, e_{i+1}/h_{i+1}, \ldots, s\, e_k /h_{k}\},
\end{equation*}
for every $s\geq 0$. Let $u_k^h=\Ind_{[0,1]^k}$. 
Note, that $u_i^h$ does not depend on $h_j$ for $0\leq j\leq i$. We use induction on $i$ to show that $u_i^h\in\C{k}$ and that
\begin{equation*}
\oZ(u_i^h+t) = \frac{(-1)^i}{k! \, h_{i+1}\cdots h_{k}} \psi^{(i)} (t),
\end{equation*}
for every $t\in\R$ and $0\leq i \leq k$, where $\psi^{(i)}(t) =\tfrac{\d^i}{\d t^i} \psi(t)$. 
\goodbreak

For $i=0$, set $P_h=\conv\{0,e_1/h_1,\ldots,e_k/h_k\}\in \cP_0^k$ and note that $u_0^h=\l_{P_h} \in\C{k}$. Hence, by the assumption on $\oZ$, we have
\begin{equation*}
\oZ(u_0^h+t) = \oZ(\l_{P_h}+t) = \psi(t) V_k(P_h) = \frac{1}{k!\, h_1\cdots h_k} \psi(t).
\end{equation*}
Now assume that the statement holds true for $i\geq 0$. Define the function $v_{i+1}^h$ by
\begin{equation*}
\{v_{i+1}^h\leq s\} = \{u_i^h\leq s\} \cap \{x_{i+1}\leq 1\},
\end{equation*}
for every $s\in\R$. Since $\epi v_{i+1}^h = \epi u_i^h \cap \{x_{i+1}\leq 1\}$, it is easy to see that $v_{i+1}^h\in\C{k}$. As $h_{i+1}\to 0$, we have epi-convergence of $v_{i+1}^h$ to $u_{i+1}^h$. Lemma \ref{thm:g_comp} implies that $u_{i+1}^h$ is a convex function and hence $u_{i+1}^h \in\C{k}$. Now, let $\tau_{i+1}$ be the translation $x\mapsto x+e_{i+1}$. Note that
\begin{equation*}
\{v_{i+1}^h \leq s\} \cup \{(u_i^h\circ \tau_{i+1}^{-1} + h_{i+1} )\leq s\} = \{u_i^h \leq s\},
\end{equation*} 
\begin{equation*}
\{v^h_{i+1} \leq s\} \cap \{(u_i^h\circ \tau_{i+1}^{-1} + h_{i+1} )\leq s\} \subset \{x_{i+1}=1\},
\end{equation*}
for every $s\in\R$. Since $\oZ$ is a continuous, translation invariant valuation and $\oZ(\l_P+t)=0$ for $P\in\cP_0^k$ with $\dim(P)<k$, Lemma \ref{le:reduction} and its proof imply that $\oZ$ vanishes on all functions $u\in\C{k}$ with $\dom u\subset H$, where $H$ is a hyperplane in $\R^k$. Hence, 
$$\oZ(v_{i+1}^h \vee (u_i^h\circ \tau_{i+1}^{-1} + h_{i+1}))=0.$$
Thus, by the valuation property
\begin{equation*}
\oZ(u_i^h+t) = \oZ((v_{i+1}^h+t) \wedge (u_i^h\circ \tau_{i+1}^{-1} + h_{i+1}+t)) = \oZ(v_{i+1}^h+t)+\oZ(u_i^h\circ \tau_{i+1}^{-1} + h_{i+1}+t).
\end{equation*}
Using the induction assumption and the translation invariance of $\oZ$, we obtain
\begin{equation*}
\oZ(v_{i+1}^h+t) = \frac{(-1)^{i+1}}{k! \, h_{i+2}\cdots h_k} \frac{\psi^{(i)}(t+h_{i+1})-\psi^{(i)}(t)}{h_{i+1}}.
\end{equation*}
As $h_{i+1}\to 0$, the continuity of $\oZ$ shows that
\begin{equation*}
\oZ(u_{i+1}^h+t) = \frac{(-1)^{i+1}}{k! \, h_{i+2}\cdots h_k}\lim_{h_{i+1}\to 0^+} \frac{\psi^{(i)}(t+h_{i+1})-\psi^{(i)}(t)}{h_{i+1}}.
\end{equation*}
Hence $\psi^{(i)}$ is differentiable from the right. Similarly, we have $v_{i+1}^h+t-h_{i+1} \eto u_{i+1}^h$ as $h_{i+1} \to 0$ and thus
\begin{equation*}
\oZ(u_{i+1}^h+t)=\lim_{h_{i+1}\to 0^+} \oZ(v_{i+1}^h+t-h_{i+1}) = \frac{(-1)^{i+1}}{k! \, h_{i+2}\cdots h_k} \lim_{h_{i+1}\to 0^+} \frac{\psi^{(i)}(t)-\psi^{(i)}(t-h_{i+1})}{h_{i+1}},
\end{equation*}
which shows that $\psi^{(i)}$ is differentiable from the left and therefore,
\begin{equation*}
\oZ(u_{i+1}^h+t) = \frac{(-1)^{i+1}}{k! \, h_{i+2}\cdots h_k} \psi^{(i+1)}(t),
\end{equation*}
for every $t\in\R$.
\end{proof}

\begin{lemma}
\label{le:sln_grwth_relation}
If $\,\oZ:\CV\to \R$ is a continuous, $\sln\!$ and translation invariant valuation, then the growth functions $\psi_0$ and $\zeta_0$ coincide and
\begin{equation*}
\zeta_n(t)=\frac{(-1)^n}{n!}\frac{\d^n}{\d t^n}\psi_n(t),
\end{equation*}
for every $t\in\R$.
\end{lemma}
\begin{proof}
Since $\l_{\{0\}}=\Ind_{\{0\}}$,  Lemma \ref{le:sln_grwth} implies that
\begin{equation*}
\psi_0(t)=\oZ(\l_{\{0\}}+t)=\oZ(\Ind_{\{0\}}+t)=\zeta_0(t),
\end{equation*}
for every $t\in\R$.

Now define $\bar{\oZ}:\CV\to\R$ as
\begin{equation*}
\bar{\oZ}(u)=\oZ(u)-\zeta_0\big( \min\nolimits_{x\in\Rn} u(x)\big).
\end{equation*}
By Lemma  \ref{le:min_is_a_val}, the functional $\bar{\oZ}$ is a continuous, $\sln$ and translation invariant valuation that satisfies
\begin{equation*}
\bar{\oZ}(\l_P+t)=\psi_n(t)V_n(P)
\end{equation*}
and
\begin{equation*}
\bar{\oZ}(\Ind_P+t)=\zeta_n(t)V_n(P),
\end{equation*}
for every $P\in\cP_0^n$ and $t\in\R$. Hence, by Lemma \ref{le:diff_grwth_fncts},
\begin{equation*}
\zeta_n(t)= \zeta_n(t)V_n([0,1]^n) = \bar{\oZ}(\Ind_{[0,1]^n}+t)=\frac{(-1)^n}{n!}\frac{\d^n}{\d t^n}\psi_n(t),
\end{equation*}
for every $t\in\R$.
\end{proof}

\begin{lemma}
\label{le:sln_grwth_infty}
If $\,\oZ:\CV\to\R$ is  a continuous, $\sln$ and translation invariant valuation, then its cone growth function $\psi_n$ satisfies
\begin{equation*}
\lim_{t\to\infty} \psi_n(t)=0.
\end{equation*}
\end{lemma}
\begin{proof}
Let $\{e_1,e_2,\ldots,e_n\}$ be the standard basis of $\Rn$ and let
\begin{equation*}
P=\conv\{0,\tfrac{e_1+e_2}{2},e_2,e_3,\ldots,e_n\},\qquad Q=\conv\{0,e_2,e_3,\ldots,e_n\}.
\end{equation*}
For $s>0$, define $u_s\in\CV$ by its epigraph as $\epi u_s=\epi \l_P \cap \{x_1\leq \tfrac{s}{2}\}$. Note, that for $t\geq 0$ this gives $\{u_s\leq t\} = t P\cap \{x_1 \leq \tfrac{s}{2}\}$. Let $\tau_s$ be the translation $x\mapsto x+s\,\tfrac{e_1+e_2}{2}$ and define $\l_{P,s}(x)=\l_P(x)\circ \tau_s^{-1}+s$ and similarly $\l_{Q,s}(x)=\l_Q(x)\circ \tau_s^{-1}+s$. We will now show that
\begin{equation*}
u_s \wedge \l_{P,s} = \l_P \qquad u_s \vee \l_{P,s} = \l_{Q,s},
\end{equation*}
or equivalently
\begin{equation*}
\epi u_s \cup \epi \l_{P,s} = \epi \l_P \qquad \epi u_s \cap \epi \l_{P,s} = \epi \l_{Q,s},
\end{equation*}
which is the same as
\begin{equation}
\{u_s\leq t\} \cup \{\l_{P,s}\leq t\} = \{\l_P\leq t\} \qquad \{u_s\leq t\} \cap \{\l_{P,s}\leq t\} = \{\l_{Q,s}\leq t\}
\label{eq:un_in_sublvl}
\end{equation}
for every $t\in\R$. Indeed, it is easy to see, that (\ref{eq:un_in_sublvl}) holds for all $t<s$. Therefore, fix an arbitrary $t\geq s$. We  have
\begin{equation*}
\{\l_{P,s}\leq t\} = \{\l_{P}+s\leq t\}+ s\,\tfrac{e_1+e_2}{2}=(t-s) P + s\,\tfrac{e_1+e_2}{2}. 
\end{equation*}
This can be rewritten as
\begin{equation*}
\{\l_{P,s}\leq t\} = tP\cap\{x_1\geq \tfrac{s}{2}\}.
\end{equation*}
Hence
\begin{equation*}
\{u_s\leq t\}\cup \{\l_{P,s}\leq t\} = \big(tP\cap \{x_1 \leq \tfrac{s}{2}\}\big) \cup \big(tP\cap\{x_1\geq \tfrac{s}{2}\}\big) = tP = \{\l_P\leq t\},
\end{equation*}
and
\begin{align*}
\{u_s\leq t\}\cap \{\l_{P,s}\leq t\} &= tP\cap\{x_1=\tfrac{s}{2}\} = \big((t-s) P \cap \{x_1=0\}\big)+s\,\tfrac{e_1+e_2}{2}\\[4pt]
&=(t-s)\, Q +s\,\tfrac{e_1+e_2}{2}=\{\l_{Q}+s\leq t\} +s\,\tfrac{e_1+e_2}{2}=\{\l_{Q,s}\leq t\}.
\end{align*}
From the valuation property of $\oZ$ we now get
\begin{equation*}
\oZ(u_s)+\oZ(\l_{P,s}) = \oZ(\l_P)+\oZ(\l_{Q,s}).
\end{equation*}
By Lemma \ref{le:sln_grwth} and since $V_n(Q)=0$, we have
\begin{equation*}
\oZ(u_s)+\psi_n(s)V_n(P)+\psi_0(s) = \psi_n(0)V_n(P)+\psi_0(0)+\psi_0(s).
\end{equation*}
As $s\to\infty$, we obtain  $u_s \eto \l_P$ and therefore
\begin{multline*}
\psi_n(0)V_n(P)+\psi_0(0)-\psi_n(s)V_n(P)= \oZ(u_s)
\xrightarrow{s\to\infty} \quad \oZ(\l_P) = \psi_n(0)V_n(P)+\psi_0(0).
\end{multline*}
Since $V_n(P)>0$, this shows that $\psi_n(s)\to 0$.
\end{proof}

Lemma \ref{le:sln_grwth_relation} shows that for a continuous, $\sln$ and trans\-lation invariant valuation $\oZ$ the indicator growth functions $\zeta_0$ and $\zeta_n$ coincide with its cone growth function $\psi_0$  and  up to a constant factor with  the $n$-th derivative of its cone growth function $\psi_n$, respectively.   Since  Lemma \ref{le:sln_grwth_infty} shows that $\lim_{t\to\infty} \psi_n(t)=0$, the cone growth functions $\psi_0$ and $\psi_n$ are completely determined by the indicator growth functions of $\oZ$. Hence  Lemma \ref{le:sln_val_determined_by_cone_grwth_functions}  immediately implies the following result.

\begin{lemma}
\label{le:reduction2}
Every continuous, $\sln$ and translation invariant valuation $\oZ:\CV\to\R$ is uniquely determined by its indicator growth functions.
\end{lemma}

We also require the following result.

\begin{lemma}
\label{le:derivative_has_finite_moment}
Let $\zeta\in C(\R)$  have constant sign on $[t_0,\infty)$ for some $t_0\in\R$. If  there exist $n\in\N$, $c_n\in\R$ and $\psi\in C^n(\R)$  with $\lim_{t\to+\infty} \psi(t)=0$ such that 
$$\zeta(t) = c_n \,\frac{\d^n}{\d t^n}\psi(t)$$ 
for $t\ge t_0$,  then
\begin{equation*}
\Big| \int_{0}^{+\infty} t^{n-1} \zeta(t) \d t\Big| < +\infty.
\end{equation*}
\end{lemma}
\begin{proof}
Since we can always consider $\widetilde{\psi}(t) = \pm c_n\, \psi(t)$ instead of $\psi(t)$, we assume that $c_n=1$ and $\zeta \geq 0$.
To prove the statement, we use induction on $n$ and start with the case $n=1$. For $t_1>t_0$,
\begin{equation*}
\int_{t_0}^{t_1} \zeta(t) \d t = \int_{t_0}^{t_1} \psi' (t) \d t = \psi(t_1) - \psi(t_0).
\end{equation*}
 Hence, it follows from the assumption for $\psi$ that
\begin{equation*}
\int_{t_0}^{+\infty} \zeta(t) \d t   = 
 \lim_{t_1\to+\infty} \psi(t_1) - \psi(t_0) = -\psi(t_0)<+\infty.
\end{equation*}
This proves the statement for $n=1$.

Let $n\geq 2$ and assume that the statement holds true for $n-1$. Since $\zeta \geq 0$, the function $\psi^{(n-1)}$ is increasing. Therefore, the limit
\begin{equation*}
c=\lim_{t\to+\infty} \psi^{(n-1)}(t) \in (-\infty,+\infty]
\end{equation*}
exists.
Moreover, $\psi^{(n-1)}$ has constant sign on $[\bar t_0,+\infty)$ for some $\bar t_0 \geq t_0$. By the induction hypothesis,
\begin{equation*}
\left| \int_{0}^{+\infty} t^{n-2} \psi^{(n-1)} (t) \d t\right | < +\infty,
\end{equation*}
which is only possible if $c=0$. In particular, $\psi^{(n-1)}(t) \leq 0$ for all $t \geq \bar t_0$.

Using integration by parts, we obtain
\begin{equation}\label{parts}
\int_{t_0}^{t_1} t^{n-1} \psi^{(n)}(t) \d t = t_1^{n-1}\psi^{(n-1)}(t_1) - t_0^{n-1}\psi^{(n-1)}(t_0)-(n-1)\int_{t_0}^{t_1} t^{n-2} \psi^{(n-1)}(t) \d t.
\end{equation}
Since $t^{n-1} \psi^{(n)}(t)\geq 0$ for $t\geq \max\{0,t_0\}$, we have
\begin{equation*}
d= \int_{t_0}^{+\infty} t^{n-1} \psi^{(n)}(t) \d t \in (-\infty,+\infty].
\end{equation*}
Hence, (\ref{parts}) implies that $t_1^{n-1} \psi^{(n-1)}(t_1)$ converges to
\begin{equation*}
d+t_0^{n-1}\psi^{(n-1)}(t_0)+(n-1)\int_{t_0}^{+\infty} t^{(n-2)} \psi^{(n-1)}(t) \d t.
\end{equation*}
Since $t_1^{n-1} \psi^{(n-1)}(t_1) \leq 0$ for $t_1\geq \max\{\bar t_0,0\}$, we conclude that $d$ is not $+\infty$.
\end{proof}

\goodbreak

\section{Proof of the Theorem}

If  $\zeta_0: \R \to [0,\infty)$ is continuous  and  $\zeta_n: \R \to [0,\infty)$ is continuous with finite $(n-1)$-st moment, then
 Lemmas \ref{le:min_is_a_val} and \ref{le:muf_is_a_val} show that
$$u\mapsto \zeta_0\big(\min\nolimits_{x\in\Rn}u(x)\big) + \int_{\dom u} \zeta_n\big(u(x)\big)\,dx$$
defines a non-negative, continuous,  $\sln$ and translation invariant valuation on $\CV$.

\goodbreak
Conversely, let $\oZ:\CV\to [0,\infty)$ be a  continuous,  $\sln$ and translation invariant valuation on $\CV$ with indicator growth functions $\zeta_0$ and $\zeta_n$. 
For  a polytope $P\in\cP_0^n$ with $\dim P < n$, Lemma \ref{le:sln_grwth} implies that
\begin{equation*}
0\leq \oZ(\Ind_P+t)=\zeta_0(t)
\end{equation*}
for every $t\in\R$. Hence, $\zeta_0$ is a non-negative and continuous function. Similarly, for $Q\in\cP_0^n$ with $V_n(Q)>0$, we have 
\begin{equation*}
0 \leq \oZ(\Ind_{s Q}+t)=\zeta_0(t)+s^n \zeta_n(t) V_n(Q),
\end{equation*}
for every $t\in\R$ and $s >0$. Therefore, also $\zeta_n$ is a non-negative and continuous function. By Lemmas \ref{le:sln_grwth_relation},  \ref{le:sln_grwth_infty} and \ref{le:derivative_has_finite_moment}, the growth function $\zeta_n$ has finite $(n-1)$-st moment.  Finally, for $u=\Ind_P+t$ with  $P\in\cP_0^n$ and $t\in\R$, we obtain %from Lemma \ref{le:sln_grwth} 
that
$$
\oZ(u) =\zeta_0 (t) +\zeta_n V_n(P) =\zeta_0(\min\nolimits_{x\in\R^n} u(x)) + \int_{\dom u} \zeta_n(u(x)) \d x.
$$
By the first part of the proof, 
$$u\mapsto \zeta_0(\min\nolimits_{x\in\R^n} u(x)) + \int_{\dom u} \zeta_n(u(x)) \d x$$
defines a non-negative, continuous, $\sln$ and translation invariant valuation on $\CV$. Thus Lemma \ref{le:reduction2} completes the proof of the theorem.

\subsection*{Acknowledgments}
The work of Monika Ludwig and Fabian Mussnig was supported, in part, by Austrian Science Fund (FWF) Project P25515-N25.  The work of Andrea Colesanti 
was supported by the G.N.A.M.P.A. and by the F.I.R. project 2013: Geometrical and Qualitative Aspects of PDE's.

%\newpage\bibliography{../arbeiten}\bibliographystyle{../meine}

%\footnotesize

\bigskip\bigskip\footnotesize
\parindent 0pt

\parbox[t]{8.5cm}{Andrea Colesanti\\
Dipartimento di Matematica e Informatica \lq\lq U. Dini\rq\rq\\
Universit\`a degli Studi di Firenze \\
Viale Morgagni 67/A\\
50134, Firenze, Italy\\
e-mail:  colesant@math.unifi.it
}
\parbox[t]{8.5cm}{Monika Ludwig\\
Institut f\"ur Diskrete Mathematik und Geometrie\\
Technische Universit\"at Wien\\
Wiedner Hauptstra\ss e 8-10/1046\\
1040 Wien, Austria\\
e-mail: monika.ludwig@tuwien.ac.at}

\bigskip\smallskip

\parbox[t]{8.5cm}{
Fabian Mussnig\\
Institut f\"ur Diskrete Mathematik und Geometrie\\
Technische Universit\"at Wien\\
Wiedner Hauptstra\ss e 8-10/1046\\
1040 Wien, Austria\\
e-mail: fabian.mussnig@tuwien.ac.at}

\end{document}